\documentclass[11pt, leqno]{amsart}

\usepackage{mathrsfs}
\usepackage{graphicx}
\usepackage{amsfonts,delarray,amssymb,amsmath,amsthm,a4,a4wide}
\usepackage{latexsym}
\usepackage{epsfig}
\usepackage{color}

\newtheorem{thm}{Theorem}[section]
\newtheorem{cor}[thm]{Corollary}

\newtheorem{lem}[thm]{Lemma}
\newtheorem{prop}[thm]{Proposition}
\newtheorem{rem}[thm]{Remark}

\theoremstyle{definition}
\newtheorem{defn}[thm]{Definition}

\numberwithin{equation}{section}
\newcommand{\norm}[1]{\left\Vert#1\right\Vert}

\newcommand{\R}{\mathbb R}

\newcommand{\p}{\partial}
\newcommand{\dist}{\mbox{dist}\,}
\newcommand{\diam}{\mbox{diam}\,}

\newcommand{\comment}[1]{}
\def\h{\hspace*{.24in}}

\usepackage{esvect}

\makeatletter
\@namedef{subjclassname@2020}{%
  \textup{2020} Mathematics Subject Classification}
\makeatother
\begin{document} 

\title[Singular and degenerate Monge-Amp\`ere equations]{Remarks on sharp boundary estimates for singular and degenerate Monge-Amp\`ere equations}
\author{Nam Q. Le}
\address{Department of Mathematics, Indiana University, 831 E 3rd St,
Bloomington, IN 47405, USA}
\email{nqle@indiana.edu}
\thanks{The author was supported in part by the National Science Foundation under grant DMS-2054686.}


\maketitle
\begin{abstract}

By constructing appropriate smooth, possibly non-convex supersolutions, we establish sharp lower bounds near the boundary for the modulus of nontrivial solutions to singular and degenerate Monge-Amp\`ere equations of the form $\det D^2 u =|u|^q$ with zero boundary condition on a bounded domain in $\R^n$.
These bounds imply that currently known global H\"older regularity results for these equations are  optimal for all $q$ negative, and almost optimal for $0\leq q\leq n-2$. 
Our study also establishes the optimality of global $C^{\frac{1}{n}}$ regularity for convex solutions to the Monge-Amp\`ere equation with finite total Monge-Amp\`ere measure.
Moreover, when $0\leq q<n-2$, the unique solution has its gradient blowing up near any  flat part of the boundary. The case of $q$ being $0$ is related to surface tensions in dimer models. 
We also obtain new global log-Lipschitz estimates, and apply them to the Abreu's equation with degenerate boundary data. 
 \end{abstract}

\section{Introduction and  statement of the main results}
This note is concerned with sharp boundary estimates for 
 non-trivial convex Aleksandrov solutions (see Definition \ref{Aleksol}) $u\in C(\overline{\Omega})$ to the Monge-Amp\`ere equation
  \begin{equation}
  \label{uq_eq}
   \det D^{2} u~= |u|^q \h~\text{in} ~\Omega, \quad
u =0\h~\text{on}~\p \Omega
\end{equation}
where $\Omega\subset \R^n$ $(n\geq 2)$ is a bounded convex domain, and $q\in \R$.  Near the boundary, the above Monge-Amp\`ere equation is degenerate if $q>0$ while it is singular if $q<0$. 
Note that, when $q<n$,  the nontrivial solution to (\ref{uq_eq}) is unique.  This follows from \cite[Proposition 4.5]{LSNS} for $0\leq q<n$, and from \cite[Theorem 1.1]{LDCDS} for $q<0$.
Depending on the particular values of $q$, this problem appears in diverse contexts. They include: $(i)$ the affine hyperbolic sphere, the $L_p$-Minkowski problem, the Minkowski problem, and the logarithmic Minkowski problem when $q<0$ (see \cite{LDCDS} and the references therein); $(ii)$ problems related to the Monge-Amp\`ere eigenvalue problem
(see \cite{LDCDS, Tso} and the references therein) when $q>0$; $(iii)$ and surface tensions in dimer models when $q=0$ and $n=2$ which we will say more 
later in this introduction. 

\vglue 0.2cm

 If $q\geq n-2$, then by \cite[Propositions 5.3 and 5.4]{LSNS}, any solution $u$ to (\ref{uq_eq}) is almost Lipschitz globally (but it is conceivable to have better regularity), that is, for all $\alpha\in (0, 1)$, we have
 \begin{equation}
\label{uqposn2}
|u(x)| \leq C(n, q,\alpha, \diam(\Omega))\dist^{\alpha}(x,\p\Omega)\|u\|_{L^{\infty}(\Omega)} \quad\text{for all } x\in\Omega.
\end{equation}
In this note, we use $\dist(\cdot, E)$ to denote the distance function to a closed set $E\subset\R^n$. 
 Unless $q=0$ and $n=2$ (where $Du$ actually blows up for certain domains $\Omega$; see the discussion after (\ref{uqzero2})), it is still an interesting problem to see whether $Du$ blows up near the boundary in this range of $q\geq n-2$. 

On the other hand, when $q<n-2$, we will see that the unique nontrivial solution $u$ to (\ref{uq_eq}) has its gradient blowing up near any {\it flat part of the boundary}. This is a new phenomenon for $0\leq q<n-2$. Our interest here is to establish sharp blow-up rate for the gradient of solutions to (\ref{uq_eq}) near the boundary when $q<n-2$. This is equivalent to obtaining sharp lower bound for $|u|$ near the boundary. 

We recall here known upper bounds for $|u|$ when $q<n-2$. 
\begin{enumerate}
\item[(1)] If $q<0 $, then it was shown in \cite[Theorem 1.1]{LDCDS} that the solution to to (\ref{uq_eq}) satisfies
\begin{equation}
\label{uqneg}
|u(x)| \leq C(n, q,\diam(\Omega))\dist^{\frac{2}{n-q}}(x,\p\Omega) \quad\text{for all } x\in\Omega
\end{equation}
and this global H\"older regularity was shown to be optimal if $q\leq -1$. The case $-1<q<0$ was left open. 
\item[(2)] If $0<q<n-2$ and $n\geq 3$, then by \cite[Proposition 1]{LDCDS}, we have for all $\beta \in (0, \frac{2}{n-q})$ the following estimate for the non-trivial solution to (\ref{uq_eq})
\begin{equation}
\label{uqpos}
|u(x)| \leq C(n, q,\beta, \diam(\Omega))\dist^{\beta}(x,\p\Omega) \quad\text{for all } x\in\Omega.
\end{equation}
\item[(3)] If $q=0$, then Caffarelli \cite{C} showed more generally that
if 
\begin{equation}
\label{u0_eq}
\det D^2 u\leq 1\quad\text{in }\Omega,\quad\text{and } u=0\quad\text{on }\p\Omega,
\end{equation}
then
 \begin{equation}
 \label{uqzero}
 |u(x)|\leq C(n,\alpha, \diam (\Omega))\dist^{\alpha} (x,\p\Omega)  \quad\text{for } 
 \begin{cases}
\text{all } \alpha\in (0, 1)& \text{when }n=2, \\ \alpha=\frac{2}{n}&\text{when }n\geq 3
 \end{cases}
 . \end{equation} 
 Therefore, solutions to (\ref{u0_eq}) are globally almost Lipschitz in two dimensions, while in dimensions $n\geq 3$, they satisfy $u\in C^{ \frac{2}{n}}(\overline{\Omega})$. Note that (\ref{uqzero}) sharpens the classical Aleksandrov estimate (\ref{AlekMu}), which gives 
 the $C^{0, \frac{1}{n}}(\overline{\Omega})$ regularity, for convex solutions to the Monge-Amp\`ere equation with finite total Monge-Amp\`ere measure. See also the sharpness of the global $C^{0, \frac{1}{n}}$ regularity in Theorem \ref{C1nthm}.
 \end{enumerate}
 
In this note, we show that the exponent $\frac{2}{n-q}$ in (\ref{uqneg}) is also optimal for $-1<q<0$ (thus answering positively an open question in \cite{LDCDS}), the range of exponents $\beta \in (0, \frac{2}{n-q})$ in (\ref{uqpos}) is almost optimal for $q\in (0, n-2)$, and for $q=0$, 
 the exponent $\alpha=\frac{2}{n}$ in (\ref{uqzero}) is optimal in dimensions $n\geq 3$. Proposition \ref{Alek2d} and the ensuing discussion will address the case of $q=0$ and $n=2$. Our first main theorem states as follows.

\begin{thm} [Sharp boundary estimates for singular and degenerate Monge-Amp\`ere equations]
\label{minusp}
Let $\Omega$ be a bounded convex domain in $\R^n$ ($n\geq 2$)  such that there is a non-empty closed subset $\Gamma$ of its boundary $\p\Omega$ that lies on a hyperplane and $\Gamma$ contains an $(n-1)$-dimensional ball. 
Let $q<n-2$. 
Let $u\in C(\overline{\Omega})$ be the non-trivial Aleksandrov solution to
  \begin{equation}
  \label{uq_eq2}
 \left\{
 \begin{alignedat}{2}
   \det D^{2} u~&= |u|^q \h~&&\text{in} ~\Omega, \\\
u &=0\h~&&\text{on}~\p \Omega.
 \end{alignedat}
 \right.
\end{equation}
 Then, for $x\in\Omega$ sufficiently close to the interior of $\Gamma$, we have 
 \[|u(x)|\geq 
 \begin{cases}
 c(n,q, \Omega,\Gamma)\dist^{\frac{2}{n-q}} (x,\p\Omega) &\quad\text{if } q<0,\\
 c(n,\Omega,\Gamma)\dist^{\frac{2}{n}} (x,\p\Omega)& \quad\text{if } q=0\quad\text{and } n\geq 3,\\
 c(n,q, \beta, \Omega,\Gamma)\dist^{\beta} (x,\p\Omega)&\quad\text{if } 0<q<n-2, \quad\text{and for any } \beta\in (\frac{2}{n-q}, 1) .
 \end{cases}
 \]
\end{thm}
\begin{rem} Theorem \ref{minusp} implies that,
in the presence of a flat part $\Gamma$ on the boundary $\p\Omega$, the non-trivial solution to (\ref{uq_eq2}) has unbounded gradient near $\Gamma$ for all $0<q<n-2$. This is in quite contrast to the case where $\p\Omega$ is smooth and uniformly convex. In the latter case, by \cite[Theorem 1.6]{LSNS}, we have $u\in C^{\infty}(\overline{\Omega})$ when $q$ is a positive integer, and $u\in C^{2+[q], \gamma}(\overline{\Omega})$ for some $\gamma(n, q)\in (0, 1)$ when $q$ is not an integer. Here, $[q]$ denotes the greatest integer not exceeding $q$.
\end{rem}
We comment briefly on the proof of Theorem \ref{minusp} to be given in Section \ref{minusp_Sect}.
The proof of Theorem \ref{minusp} for the case $q\leq 0$ is based on constructions of suitable supersolutions $v$ to (\ref{uq_eq2}) as in \cite{LDCDS}. The key difference here is that we use smooth, {\it possibly non-convex}  supersolutions in Lemma \ref{suplem},  thus overcoming the restriction $q\leq -1$ imposed in the construction of convex supersolutions in \cite[Lemma 2.4]{LDCDS}. 
The underlying idea is this: to obtain a lower bound for $|u|$, we mostly use that $u$ is a smooth subsolution in the interior so it is also a viscosity subsolution. The smooth supersolutions $v$ act as test functions in the definition of viscosity subsolution property for $u$ so they are not required to be convex; see the discussion in \cite[Section 3]{H}.  

In general, we cannot have a comparison principle for non-convex solutions to Monge-Amp\`ere equations. However, in the present situation, we take advantage of the fact
that  our solutions are also smooth in the interior, and the proof of the comparison principle goes through directly without relying on viscosity solutions. The proof of Theorem \ref{minusp} for the case $0<q<n-2$ is based on an iterative argument using non-convex smooth supersolutions to $\det D^2 u =\dist^\gamma(x,\p\Omega)$ where $\gamma>0$; see Lemma \ref{supdist}. We believe that the idea of using non-convex smooth supersolutions will find interesting applications in many other Monge-Amp\`ere type equations.

A a byproduct, the study of sharp global H\"older regularity for (\ref{uq_eq}) when $q<0$ allows us to prove the optimality of global $C^{0, \frac{1}{n}}$ regularity (that is, Lipschitz when $n=1$ and $C^{\frac{1}{n}}$ when $n\geq 2$) for convex solutions to the Monge-Amp\`ere equation with finite total Monge-Amp\`ere measure and zero boundary condition. 
The optimality is obvious in dimension $n=1$ with the convex function $u(x)=|x|-1$ on $\Omega= (-1, 1)\subset\R$.
In dimensions $n\geq 2$, we have the following theorem.
 \begin{thm} [Optimality of global $C^{\frac{1}{n}}$ regularity]
\label{C1nthm}
Let $n\geq 2$. Then, for any $1/n<\alpha<1$, there exist a bounded convex domain $\Omega\subset\R^n$ and a convex function $v\in C(\overline{\Omega})\cap C^{\infty}(\Omega)$ such that $v=0$ on $\p\Omega$ and
$\det D^2 v\in L^1(\Omega)$, but $v\not \in C^{\alpha} (\overline{\Omega})$.
\end{thm}
The proof of Theorem \ref{C1nthm} will be given in Section \ref{C1n_sect}. We note that, in dimensions $n\geq 5$, the optimality of global $C^{\frac{1}{n}}$ regularity was also obtained by Du \cite[Theorem 3.2]{D}. Du's method is a bit involved and does not seem to extend to cover the lower dimensional cases.

When $q=0$ and $n=2$, the almost Lipschitz estimate (\ref{uqzero}) is almost sharp. We will prove
 the following proposition whose consequence strengthens Caffarelli's estimate in two dimensions from global almost Lipschitz to global log-Lipschitz.
\begin{prop}[Global log-Lipschitz estimates]
\label{Alek2d}
Let $\Omega\subset\R^n$ be a bounded convex domain. Let $u\in C(\overline{\Omega})$ be a convex function satisfying $u=0$ on $\p\Omega$, and 
\[\det D^2 u\leq M \dist^{n-2}(\cdot,\p\Omega) \quad\text{in }\Omega\] in the sense of Aleksandrov. Then \[|u(x)|\leq C(\Omega, M) \dist (x,\p\Omega) (1+ |\log \dist(x,\p\Omega)|) \quad \text{for all }x\in\Omega.\] 
\end{prop}
Combining Proposition \ref{Alek2d} with (\ref{uqposn2}), we obtain  the global log-Lipschitz property for solutions to
\[\det D^2 u =m |u|^q\quad\text{in }\Omega, u=0\quad\text{on }\p\Omega \]
when $q>n-2$. The case $q=n$ corresponds to the Monge-Amp\`ere eigenvalue problem. We state this as a corollary.
\begin{cor} Let $\Omega\subset\R^n$ $(n\geq 2)$ be a bounded convex domain. Suppose that $u\in C(\overline{\Omega})$ is a convex solution to 
\[\det D^2 u =m |u|^q\quad\text{in }\Omega, u=0\quad\text{on }\p\Omega \]
where $q>n-2$ and $m>0$. Then
 \[|u(x)|\leq C(\Omega, n, q, m) \dist (x,\p\Omega) (1+ |\log \dist(x,\p\Omega)|) \|u\|_{L^{\infty}(\Omega)}\quad \text{for all }x\in\Omega.\] 
\end{cor}

An immediate consequence of Proposition \ref{Alek2d} is the following result:  if $\Omega\subset\R^2$ is a bounded convex domain, and if $u\in C(\overline{\Omega})$ is a convex function satisfying $\det D^2 u\leq 1$ in $\Omega$ with $u=0$ on $\p\Omega$, then 
we have the following log-Lipschitz estimate
\begin{equation}
\label{uqzero2}
|u(x)|\leq C(\Omega) \dist (x,\p\Omega)  |\log \dist(x,\p\Omega)| \quad \text{for all }x\in\Omega\text{ close to }\p\Omega.
\end{equation}
The estimate (\ref{uqzero2}) turns out to be sharp. Its sharpness can be seen from the explicit formula for the solution to the Monge-Amp\`ere equation $\det D^2 u=1$ with zero boundary data on a planar triangle; see (\ref{sigL}). Interestingly, this explicit formula comes from 
the study of dimer models in combinatorics and statistical physics in the works of Cohn, Kenyon, Okounkov, Propp, Sheffield \cite{CKP, K, KO, KOS}, among others.  The particular case of the triangle comes from the study of the lozenge tiling model, and from (\ref{sigL}), one deduces that the gradient of the surface tension of  the lozenge tiling model 
blows up with a rate proportional to the logarithm of the distance to the boundary. This turns out to be true near a generic boundary point for general surface tensions; see Proposition \ref{Dsiglog}.

We end this introduction by discussing an application of
the log Lipschitz estimate in Proposition \ref{Alek2d} is to  
the Abreu's equation \cite{Ab} with degenerate boundary data. This is a fourth order fully nonlinear partial differential equation of the form
\begin{equation}
\label{AMCE}
\left \{
\begin{alignedat}{2}
U^{ij}D_{ij}w&=-f&&\quad \text{in}~\Omega,\\
w &= [\det D^{2} u]^{-1}&&\quad \text{in}~\Omega,\\
w &= 0 &&\quad \text{on}~\p\Omega,
\end{alignedat}
\right.
\end{equation}
where $u$ is a locally uniformly convex function defined in a bounded domain $\Omega\subset\R^{n}$. Throughout, $U = (U^{ij})=(\det D^{2} u) (D^{2} u)^{-1}$ is the cofactor matrix of the 
Hessian matrix $D^{2}u$.
When $\Omega$ is the interior of a bounded polytope, equation (\ref{AMCE})
arises in the study of the existence of constant scalar curvature K\"ahler metrics for toric varieties \cite{D1}, and the function $f$ corresponds to the scalar curvature of the toric varieties. When $\Omega$ is uniformly convex and smooth, Chen-Li-Sheng \cite{CLS} proved the existence of a smooth and strictly convex solution $u$ in $\Omega$  for equation (\ref{AMCE}) with boundary data of the form
  \begin{equation}
\label{SBV}
u=\varphi,~~~|Du|=\infty, ~~~ w= 0~~~\text{on}~\p\Omega.
\end{equation}
where the function $\varphi$ is assumed to be smooth and uniformly convex in an open set of $\R^n$ containing $\overline{\Omega}$, and $f$ is smooth with a positive lower bound. 

The vanishing of $w$ on $\p\Omega$ in (\ref{AMCE}) causes the Hessian determinant of $u$ to be infinite on $\p\Omega$. If 
\begin{equation}
\label{d1a}
\det D^2 u(x)\text{ grows at least as }\dist^{-1}(x, \p\Omega) \text{ when }x  \text{ tends to the boundary},
\end{equation}
then in \cite[Theorem 5]{D2}, Donaldson established the following 
global lower bound for the Hessian determinant of a solution to (\ref{AMCE}): 
for all $\alpha\in (0, 1)$, we have
\begin{equation} 
\label{Abwa}
\det D^2 u (x)\geq C(n, \alpha) d(x)^{-\alpha}  \|f^{+}\|_{L^{\infty}(\Omega)}^{-n}
      [\diam(\Omega)]^{2n+\alpha}~\text{for all  }x\in\Omega.
      \end{equation}
The exponent $\gamma$ in $ \|f^{+}\|_{L^{\infty}(\Omega)}^{\gamma}$  was written as $n$ in \cite{D2} but it should be $-n$ as can be seen from inequality (35) in [D2]; see also the rescaling (\ref{u_res}).

In the following theorem, we sharpen the estimate (\ref{Abwa}) without requiring the asymptotic behavior (\ref{d1a}). 

\begin{thm} [Log-Lipschitz estimates for the inverse of the Hessian determinant of Abreu's equation with degenerate boundary data]
\label{wphi}
Let $\Omega\subset\R^n$ be a bounded domain. Let $u\in C^{4}(\Omega)$ be a solution of (\ref{AMCE}) where $f\in L^{\infty}(\Omega)$ with $f^{+}>0$, and $w\in C(\overline{\Omega})$.
 Then
\begin{equation}
\label{Awlog}
\det D^2 u\geq c(n, \diam (\Omega)) \|f^{+}\|_{L^{\infty}(\Omega)}^{-n} \dist^{-1} (x,\p\Omega)  (1+ |\log \dist(x,\p\Omega)|)^{-1}.
\end{equation}
\end{thm}

The rest of the note is organized as follows.  In Section \ref{MA_sect}, we recall the notion of Aleksandrov solutions to the Monge-Amp\`ere equation, and an explicit formula for the surface tension on the triangle.  In Section \ref{C1n_sect}, we prove Theorem \ref{C1nthm}.
 In Section \ref{minusp_Sect}, we prove Theorem \ref{minusp}. In Section \ref{subsup}, we give the proofs of  Proposition \ref{Alek2d} and Theorem \ref{wphi}.  In Section \ref{sfrem_sect}, we show in Proposition \ref{Dsiglog} that the gradient of general surface tensions behaves similarly to that of the lozenge tiling model. 
\section{Aleksandrov solutions, and surface tension}
\label{MA_sect}
In this section, we recall the notion of Aleksandrov solutions to the Monge-Amp\`ere equation, Aleksandrov estimates, and an important explicit formula for the surface tension on the triangle.
\subsection{The Monge-Amp\`ere measure}
  Let $u:\Omega\rightarrow\R$ be a convex function defined on a bounded convex domain $\Omega\subset\R^n$.

 Let $Mu$
be the Monge-Amp\`ere measure associated with $u$; see \cite[Definition 2.1 and Theorem 2.3]{F}, \cite[Theorem 1.1.13]{G}, and \cite[Definition 3.4]{LMT}. It is defined by
$$Mu(E) = |\p u(E)|~\text{where } \p u(E) = \bigcup_{x\in E} \p u(x),~\text{for each Borel set } E\subset\Omega$$
where
$$
\partial u (x):=\{p\in \R^{n}\,:\, u(y)\ge u(x)+p\cdot (y-x)\quad \text{for all}\, y \in \Omega\}.
$$
Note that, when $u\in C^2 (\Omega)$, we have $\p u(x)=\{Du(x)\}$ for all $x\in\Omega$, and $Mu$ is absolutely continuous with respect to the Lebesgue measure, and it has density $\det D^2 u$. In this note, we use $\det D^2 u$ to denote $Mu$ for $u$ being merely convex.

\begin{defn}[Aleksandrov solutions]\label{Aleksol} Given a Borel measure \(\mu\) on \(\Omega\), a convex
function \(u:\Omega \to \R\) is called an \emph{Aleksandrov solution} to the Monge-Amp\`ere equation
\[
\det D^{2} u =\mu,
\]
if $\mu=Mu$ as Borel measures.  When $\mu=f~dx $ we will simply say that $u$ solves 
\begin{equation*}
\det D^2 u=f
\end{equation*}
and this is the notation we use in this note.
Similarly, when writing  $\det D^2 u \geq  \lambda \ (\leq \Lambda)$ we mean that $Mu \ge \lambda \,dx \ (\leq \Lambda\,dx)$.
\end{defn}
  
 We recall the following comparison principle; see \cite[Theorem 2.10]{F}, \cite[Theorem 1.4.6]{G}, and \cite[Lemma 3.25]{LMT}.

\begin{lem}[Comparison principle]
\label{comp_prin}
 Let  $\Omega\subset\R^n$ be an open, bounded and convex domain. Let $u, v\in C(\overline{\Omega})$ be convex functions.
If $u\ge v$ on $\p \Omega$, and  
$$
\det D^{2}u \leq \det D^{2}v\quad \text{in}~\Omega, 
$$
then $u\ge v$ in $\Omega$.
\end{lem}

The classical Aleksandrov maximum principle for the Monge-Amp\`ere equation (see \cite[Theorem 2.8]{F}, \cite[Theorem 1.4.2]{G} and \cite[Theorem 3.20]{LMT}) states that: if $u\in C(\overline{\Omega})$ is a convex function satisfying
\begin{equation*} Mu(\Omega)=\int_\Omega \det D^2 u~dx:=m<\infty,\quad\text{and }  u=0 \quad\text{on }\p\Omega,\end{equation*}
then
\begin{eqnarray}
\label{AlekMu}
|u(x)| &\leq& C(n, \diam(\Omega))[\dist(x, \p\Omega)]^{1/n} \Big(\int_\Omega \det D^2 u dx\Big)^{1/n}\nonumber \\ &\leq& C(n, m, \diam(\Omega))[\dist(x, \p\Omega)]^{1/n}.
\end{eqnarray}
In particular, by the convexity of $u$, one finds that $u\in C^{0, 1/n}(\overline{\Omega})$. Theorem \ref{C1nthm} asserts the optimality of this global $C^{\frac{1}{n}}$ regularity in all dimensions $n\geq 2$.

As mentioned in the Introduction, Caffarelli's estimate (\ref{uqzero}) sharpens (\ref{AlekMu}) in the special case $\det D^2 u\leq 1$. In two dimensions, we can further sharpen Caffarelli's estimate using Proposition \ref{Alek2d}.

 \subsection{Surface tension for the lozenge tiling model} 
\label{lozenge_sect}
Let $T$ be the triangle with vertices at $(0,0),  (1,0)$ and $(0, 1)$, and let $T^0$ be the interior of $T$. Then, there is a remarkable formula for 
 the solution
  to
  \begin{equation}
 \label{STET}
 \left\{
 \begin{alignedat}{2}
   \det D^{2} \sigma_T~&= 1 \h~&&\text{in} ~T^o, \\\
\sigma_T &=0\h~&&\text{on}~\p T.
 \end{alignedat}
 \right.
\end{equation}
The solution is called the surface tension in the study of the lozenge tiling model in the works \cite{CKP, KO, KOS}; see 
also \cite[Theorem 8]{K} and \cite[Lecture 12]{Go}. It is given by
  \begin{equation}
  \label{sigL}
  \sigma_T (x) =-\frac{1}{\pi^2} (\mathscr{L} (\pi x_1)+ \mathscr{L} (\pi x_2) + \mathscr{L} (\pi (1-x_1-x_2)))\quad\text{for } x=(x_1, x_2)\in T^o
  \end{equation}
  where
  \[\mathscr{L} ( \theta)=-\int_0^\theta \log |2\sin u| d u\]
  is the Lobachevsky function.
  
  One notes that
  \begin{equation}
  \label{Dsig}
  D\sigma_T (x) =\frac{1}{\pi} \Big(\log \Big(\frac{\sin (\pi x_1)}{\sin (\pi(x_1+ x_2))}, \log \Big(\frac{\sin (\pi x_2)}{\sin (\pi(x_1+ x_2))}\Big)\Big).
  \end{equation}
 A direct calculation shows that $\det D^2\sigma_T=1$ in $T^o$. To verify the boundary condition, we use the fact that \[\mathscr{L}(\pi)=-\int_0^\pi \log |2\sin u| d u=0;\]
  see \cite[p. 161]{Ahlf}. From this, we see that $\mathscr{L}$ is odd, and $\pi$-periodic. The boundary condition in (\ref{STET}) can be then easily verified.
 
\section{Optimality of global  $C^{\frac{1}{n}}$ regularity and singular Monge-Amp\`ere }
\label{C1n_sect}
In this section, we prove Theorem \ref{C1nthm}. 
Interestingly, our construction of the bounded domain and the convex function in Theorem \ref{C1nthm} is inspired by the study of 
 the optimal global H\"older regularity  of the unique convex solution $u\in C^{\infty}(\Omega)\cap C(\overline{\Omega})$ to singular Monge-Amp\`ere equations of the form (\ref{uq_eq}) with $q=-p<0$, where $p>0$. 
In \cite[Theorem 1.1]{LDCDS}, it was shown that $u\in C^{\frac{2}{n+p}}(\overline{\Omega})$ and this global H\"older regularity is optimal for all $p\geq 1$ (it is in fact optimal for all $p>0$ by Theorem \ref{minusp}). For any $\alpha>1/n$, we can choose $1\leq p<n$ such that $\alpha>\frac{2}{n+ p}$, and the construction of the convex supersolutions to (\ref{uq_eq}) in \cite{LDCDS} gives a candidate for our desired
function $v$. Here, we only need $\det D^2 v\in L^1(\Omega)$ but do not require $\det D^2 v\leq |v|^{-p}$ in $\Omega$ so our construction actually works for all $p>0$.

For $x=(x_1, \cdots, x_n)\in\R^n$, let us write $x=(x', x_n)$, where $x'=(x_1, \cdots, x_{n-1})\in\R^{n-1}$. Our basic construction is the following:
\begin{lem} 
\label{pnlem} Let $n\geq 2$. Let $0< p<n$, and 
 $$\Omega=\{(x', x_n): |x'|<1, 0<x_n < 1- |x'|^2\}.$$
Then  the function $$w= x_n -x_n^{\frac{2}{n+ p}} (1-|x'|^2)^{\frac{n+ p-2}{n+ p}}$$ has the following properties:
\begin{enumerate}
 \item[(i)] $w$ is smooth, convex in $\Omega$ with $w=0$ on $\p\Omega$,
 \item[(ii)] $\det D^2 w\in L^1(\Omega)$,
 \item[(iii)] $w\not \in C^{\alpha} (\overline{\Omega})$ for any $\alpha>\frac{2}{n+ p}$.
 \end{enumerate}
\end{lem}

\begin{proof} We use the calculation carried out in the proof of \cite[Lemma 2.4]{LDCDS}.  For $x=(x', x_n)$, we denote $r=|x'|$. Let
$$a= \frac{2}{n+ p},\quad \text{and }b= 1-a=\frac{n+ p-2}{n+ p}.$$
Then
$$w(x)=x_n -x_n^a (1-r^2)^b.$$
We calculate
\begin{align*}
w_{r}&= 2b x_n^a (1-r^2)^{b-1} r\\
w_{rr}&= 2b x_n^a (1-r^2)^{b-2}[1-(2b-1) r^2]\\
w_{x_n}&=1 -a x_n^{a-1}(1-r^2)^b\\
w_{x_n x_n}&= a(1-a) x_n^{a-2}(1-r^2)^b\\
w_{x_n r} &=2ab x_n^{a-1} (1-r^2)^{b-1} r.
\end{align*}
In suitable coordinate systems, such as cylindrical in $x'$, the Hessian of $w$ has the following form
\begin{equation*}
 D^2 w =
 \begin{pmatrix}
  \frac{w_r}{r} & 0 & \cdots & 0 & 0 \\
  0 & \frac{w_r}{r} & \cdots & 0 & 0 \\
  \vdots  & \vdots  & \ddots & \vdots & \vdots  \\
  0 & 0 & \cdots & w_{rr} & w_{r x_n}\\
  0 & 0 & \cdots & w_{r x_n} & w_{x_n x_n}
 \end{pmatrix}.
\end{equation*}
We have
\begin{eqnarray*}\det D^2 w&=&(\frac{w_r}{r})^{n-2} [w_{x_n x_n}w_{rr}- w^2_{x_n r}]\\&=&   (2b)^{n-1} x_n^{na-2}(1-r^2)^{n(b-1)} a [1-a + (1-2b-a)r^2]\\&=&
ab(2b)^{n-1} x_n^{na-2}(1-r^2)^{1-na}.
\end{eqnarray*}
Since $0<a, b<1$ and $0\leq r<1$,  we deduce that $w$ is smooth and convex in $\Omega$ with $w=0$ on $\p\Omega$. This confirms $(i)$. For $(ii)$, we note  from $0<p<n$ that \[na-1>0.\]
Thus
\begin{eqnarray*}
\int_{\Omega} \det D^2 w \,dx&=&\int_{|x'|<1}\int_0^{1-|x'|^2} ab(2b)^{n-1} x_n^{na-2}(1-|x'|^2)^{1-na} \,dx_n dx'\\&=&
\int_{|x'|<1}\frac{ ab(2b)^{n-1}}{na-1} (1-|x'|^2)^{na-1} (1-|x'|^2)^{1-na} \,dx'\\
&=& \int_{|x'|<1}\frac{ ab(2b)^{n-1}}{na-1} \,dx'<\infty.
\end{eqnarray*}
Therefore, $\det D^2 w\in L^1(\Omega)$ and this proves $(ii)$.

Finally, we prove $(iii)$. For $x=(0, x_n)$, we have
\[|w(0, x_n)|= x_n^{\frac{2}{n+ p}}- x_n\geq x_n^{\frac{2}{n+ p}}/2 \]
for $x_n$ small, depending only on $n$ and $p$. This shows that $w\not \in C^{\alpha} (\overline{\Omega})$ for any $\alpha>\frac{2}{n+ p}$.
\end{proof}
\begin{rem}
\label{sup_rem}
If $1\leq p<n$, then as in \cite[Lemma 2.4]{LDCDS}, we can find a constant $C(n, p)>0$ such that $\bar w:= Cw$ is a supersolution of (\ref{uq_eq}) where $q=-p$, that is
\[\det D^2 \bar w\leq  |\bar w|^{-p}\quad\text{in }\Omega,\quad \text{and }\bar w=0 \quad\text{on }\p \Omega.\]
\end{rem}
\begin{proof}[Proof of Theorem \ref{C1nthm}] For $n\geq 2$ and $1/n<\alpha<1$, let us choose $p$ such that \[\max\{2/\alpha-n,1\}<p<n.\]
Then
\[\alpha>\frac{2}{n+ p}.\]
Let $\Omega$ and $w$ be as in Lemma \ref{pnlem}. Then, $v=w\in C(\overline{\Omega})\cap C^{\infty}(\Omega)$ is a convex function with $v=0$ on $\p\Omega$ and
$\det D^2 v\in L^1(\Omega)$, but $v\not \in C^{\alpha} (\overline{\Omega})$.

By Remark \ref{sup_rem}, there is $C'(n, p)>0$ such that $\det D^2 v\leq C'(n, p)|v|^{-p}$ in $\Omega$.
\end{proof}

\section{Proof of Theorem \ref{minusp}}
\label{minusp_Sect}
In this section, we prove Theorem \ref{minusp}. We write a typical point $x=(x_1,\cdots, x_n)\in \R^n$ as $x= (x', x_n)$ where $x'=(x_1, \cdots, x_{n-1})$.

To prove the sharp estimate in Theorem \ref{minusp} for the case $q<0$, we will use the following  smooth, possibly non-convex supersolutions.
\begin{lem}[Possibly non-convex supersolution for negative power]
\label{suplem}
Assume $p>0$. 
Let $$\Omega_1:=\{(x', x_n): |x'|<1, 0<x_n < (1- |x'|^2)^{\frac{n}{n+ p-2}}\}.$$
Then the function $$w=\Big [x_n -x_n^{\frac{2}{n+ p}} (1-|x'|^2)^{\frac{n}{n+ p}}\Big]/2$$ is smooth in $\Omega_1$ and satisfies
$$\det D^2 w\leq |w|^{-p}\quad\text{in } \Omega_1,\quad\text{and } w=0\quad\text{on }\p\Omega_1.$$
\end{lem}
\begin{proof} For $x=(x', x_n)$, we denote $r=|x'|$.
Let
$$v(x)= -Cx_n^a (1-r^2)^b$$
where $0<a, b<1$ and $C>0$. Then, from the computations in \cite[Lemma 2.4]{LDCDS}, we have
\begin{eqnarray*}\det D^2 v&=&(\frac{v_r}{r})^{n-2} [v_{x_n x_n}v_{rr}- v^2_{x_n r}]\\&=&  C^n (2b)^{n-1} x_n^{na-2}(1-r^2)^{n(b-1)} a [1-a + (1-2b-a)r^2].
\end{eqnarray*}
We would like to have $$\det D^2 v \leq |v|^{-p}= C^{-p} x_n^{-ap} (1-r^2)^{-bp}$$
which is equivalent to
\begin{equation}
\label{abC}
C^{n+p} (2b)^{n-1} x_n^{(n+ p)a-2}(1-r^2)^{(n + p)b-n } a [1-a + (1-2b-a)r^2] \leq  1 \quad\text{for all } r<1.
\end{equation}
We can choose
$$a= \frac{2}{n+ p},\quad b =\frac{n}{n+ p},\quad C= \frac{1}{2}.$$
Now, let
$$w= Cx_n + v= [x_n -x_n^a (1-r^2)^b]/2.$$
Then $w<0$ is smooth  in $\Omega_1$, and $w=0$ on $\p\Omega_1$. Moreover, since
$|w|= |v|-Cx_n,$
we have
$$\det D^2 w= \det D^2 v\leq |v|^{-p}=|Cx_n + |w||^{-p}\leq |w|^{-p}\quad\text{in }\Omega_1.$$
\end{proof}
\begin{rem} For $a$ and $b$ in the above proof, we have
\[1-a + (1-2b-a)r^2 \rightarrow \frac{2(p-2)}{n+p}\quad\text{when } r\rightarrow 1^{-}.\]
Thus the function $v$, or equivalently $w$ in Lemma \ref{suplem}, is convex only when $p\geq 2$. When $0<p<2$, $w$ is not convex. 
\end{rem}
For the case $0\leq q<n-2$, the solution $u$ to (\ref{uq_eq2}) is expected to grow like some power of the distance function to the boundary. The following lemma is the basic for our argument. 
\begin{lem}[Non-convex supersolution for positive power]
\label{supdist}
Let $\Omega$ be a bounded convex domain in $\R^n$ ($n\geq 3$)  such that there is a non-empty closed subset $\Gamma$  of its boundary $\p\Omega$ that lies on a hyperplane and $\Gamma$ contains an $(n-1)$-dimensional ball. 
Let $m,\delta>0$, and $0\leq \gamma<n-2$. 
Suppose that $u\in C(\overline{\Omega})\cap C^2(\Omega)$ is a convex solution  to
  \begin{equation*}
   \det D^{2} u~= f\h~\text{in} ~\Omega, \quad
u =0\h~\text{on}~\p \Omega,
\end{equation*}
where $f\in C(\overline{\Omega})$, $f\geq 0$, and $f$ satisfies \[f(x)\geq m \dist^{\gamma}(x, \p\Omega)\quad \text{for }x\in\Omega\quad\text{satisfying }\dist(x, \Gamma)\leq \delta.\]
 Then, for $x\in\Omega$ sufficiently close to the interior of $\Gamma$, we have 
 \[|u(x)| \geq c(n,m, \gamma,\delta, \Omega,\Gamma) \dist^{\frac{2+\gamma}{n}}(x, \p\Omega). \]
\end{lem}
\begin{proof} Let $\alpha = (2+\gamma)/n$. Then $\alpha\in (0, 1)$. 
Fix $z\in \Omega$ being close to the interior of $\Gamma$.

By translating and rotating coordinates, we can assume that for some $t=t(n, m, \gamma, \delta, \Omega,\Gamma)\in (0, \min\{2^{-n/2}m^{1/2}, \delta^{\frac{1-\alpha}{2}}\})$,
\[z=(0, z_n)\in \Omega_{t}:=\{(x', x_n): |x'|<t, 0<x_n< (t^2-|x'|^2)^{\frac{1}{1-\alpha}}\}\subset\Omega,\] \[ \{(x', 0):|x'|\leq t\}\subset\Gamma\subset\p\Omega,
\quad\text{and }\dist(x,\p\Omega) = x_n\quad\text{for } x\in\Omega_t.\]
Let 
\[v(x) = x_n + x_n^{\alpha} (|x'|^2 -t^2).\]
Then, $v\in C^{\infty}(\Omega_t)$, and $v=0$ on $\p\Omega_t$.  By a computation (see, for example, \cite[Lemma 2.2]{LDCDS}), we have
\begin{eqnarray}
\label{detvneg}
\det D^2 v &=& 2^{n-2} x_n^{n\alpha-2} \Big[\alpha(1-\alpha) t^2 -(\alpha^2 +\alpha)|x'|^2\Big]\\& \leq& 2^{n-2} t^2 x_n^{n\alpha-2} =2^{n-2} t^2 x_n^{\gamma}\leq  m x_n^{\gamma}/4.\nonumber\end{eqnarray}
By the choice of $t$, one has \[f(x) \geq m \dist^{\gamma}(x, \p\Omega)= m x_n^\gamma\quad \text{in }\Omega_t.\] It follows from $\det D^{2} u=f\geq m x_n^\gamma$ in $\Omega_t$  that
\begin{equation}
\label{detuvd}
\det D^2 v<\det D^2 u\quad\text{in }\Omega_t.\end{equation}
We show that \begin{equation}
\label{vumind}
v\geq u\quad\text{ in }\Omega_t.
\end{equation}
 Note that, $u=v$ on $\p\Omega_t\cap\Gamma$, and $u\leq 0=v$ on $\p\Omega_t\setminus\Gamma$, by the convexity of $u$. Suppose that $v<u$ somewhere in $\Omega_t$. Then $v-u$ attains its minimum value in $\overline{\Omega_t}$
at some point $\bar x\in \Omega_t$. At this point, we have $D^2 v(\bar x)\geq D^2 u(\bar x)$. Thus, by (\ref{detvneg})
\[m\bar x^{\gamma}_n/4\geq\det D^2 v(\bar x)\geq \det D^2 u(\bar x)\geq m\bar x^{\gamma}_n,\]
a contradiction to (\ref{detuvd}). Therefore, (\ref{vumind}) holds.
 Hence, for $z= (0, z_n)$, we have
\[|u(z)|\geq |v(z)| =|v(0, z_n|= z_n^{\alpha} (t^2 - z_n^{1-\alpha})\geq z_n^{\alpha} t^2/2 =\dist^{\alpha} (z,\p\Omega)t^2/2,  \]
if $z_n>0$ is sufficiently small.
\end{proof}
\begin{rem} When $|x'|\rightarrow t$, we have $[\alpha(1-\alpha) t^2 -(\alpha^2 +\alpha)|x'|^2\rightarrow -2\alpha^2 t^2<0$. Thus, from (\ref{detvneg}), we see that the function $v$ in the proof of Lemma \ref{supdist} is not convex.
\end{rem}
When $\det D^2 u\leq M \dist^{\gamma}(\cdot,\p\Omega)$, we have the following H\"older estimate whose proof is an easy adaptation of the above computation.
\begin{lem}
\label{subdist}
Let $\Omega$ be a bounded convex domain in $\R^n$ ($n\geq 3$).
Let $M>0$ and $0\leq \gamma<n-2$. Let $u\in C(\overline{\Omega})$ be an Aleksandrov solution to
  \begin{equation*}
   \det D^{2} u~\leq  M\dist^{\gamma}(\cdot, \p\Omega) \h~\text{in} ~\Omega, \quad
u =0\h~\text{on}~\p \Omega.
\end{equation*}
 Then for all $x\in\Omega$, we have 
 \[|u(x)| \leq C(n,M, \gamma,\Omega) \dist^{\frac{2+\gamma}{n}}(x, \p\Omega). \]
\end{lem}
\begin{proof}
Let $\alpha = (2+\gamma)/n$. Then $\alpha\in (0, 1)$. Let $z=(z', z_n)$ be an arbitrary point in $\Omega$. By translation and rotation of coordinates, we can assume that:  $0\in \p\Omega$,  $\Omega\subset \R^n_+=\{x=(x', x_n)\in\R^n: x_n>0\}$, 
 the $x_n$-axis points inward $\Omega$, $z$ lies on the $x_n$-axis, and $z_n=\dist (z,\p\Omega)$. 
Consider
\[w(x) = x_n + x_n^{\alpha} (|x'|^2 -K),\]
where\[ K:=\frac{2(\diam(\Omega))^2 + M+1}{\alpha(1-\alpha)}.\]
Then $w$ is convex, 
$w\in C^{\infty}(\Omega)$, $w\leq 0$ on $\p\Omega$, and 
\begin{eqnarray*}\det D^2 w(x) &=& 2^{n-2} x_n^{n\alpha-2} \Big[\alpha(1-\alpha) K -(\alpha^2 +\alpha)|x'|^2\Big]\\& \geq& (M+1) x_n^{n\alpha-2} \geq (M+1) \dist^{\gamma}(x,\p\Omega)\geq \det D^2 u.\end{eqnarray*}
By the comparison principle in Lemma \ref{comp_prin}, we have $w\leq u$. Therefore, at $z$, we have
\[|u(z)| \leq |w(z)| = K z_n^{\alpha} - z_n \leq K \dist^{\frac{2+\gamma}{n}}(z, \p\Omega). \]
\end{proof}
\begin{rem} \leavevmode
\begin{enumerate}
\item[(i)] Lemma \ref{subdist} was also established in \cite[Theorem 1.1]{LL} with a different proof.
\item[(ii)] In view of Lemma \ref{supdist}, the H\"older exponent $(2+\gamma)/n$ in Lemma \ref{subdist} is optimal.
\end{enumerate}
\end{rem}
\begin{proof}[Proof of Theorem \ref{minusp}] If $q<0$, then as observed in the proof of Theorem 1.1 in \cite{LDCDS}, we have $u\in C^{\infty}(\Omega)$ while for $q\geq 0$, we also have $u\in C^{\infty}(\Omega)$ by \cite[Proposition 2.8]{LSNS}.

{\bf Case 1: $q=-p<0$}.  Then $u\in  C^{\infty}(\Omega)\cap C(\overline{\Omega})$ solves
  \begin{equation*}
   \det D^{2} u~= |u|^{-p} \h~\text{in} ~\Omega, \quad
u =0\h~\text{on}~\p \Omega.
\end{equation*}
We need to show that 
 for $x\in\Omega$ sufficiently close to $\Gamma$, we have 
 \begin{equation}
 \label{up_est}
 |u(x)|\geq c(n,p, \Omega,\Gamma)\dist^{\frac{2}{n+p}} (x,\p\Omega).\end{equation}

First, we note that, under the following rescaling of the domain and the solution $u$:
\[\tilde \Omega:=\lambda\Omega,\quad, \tilde u(y)= \lambda^{\frac{2n}{n+p}} u(\lambda^{-1} y) (y\in\tilde\Omega),\]
$\tilde u$ solves
\begin{equation*}
 \left\{
 \begin{alignedat}{2}
   \det D^{2} \tilde u~&= |u|^{-p} \h~&&\text{in} ~\tilde \Omega, \\\
\tilde u &=0\h~&&\text{on}~\p \tilde \Omega.
 \end{alignedat}
 \right.
\end{equation*}
By rescaling using a suitable $\lambda(n, p,\Omega,\Gamma)>0$, translating and rotating coordinates, we can assume that
 $$\Omega_1:=\{(x', x_n): |x'|<1, 0<x_n < (1- |x'|^2)^{\frac{n}{n+ p-2}}\}\subset\Omega,$$
and \[\{(x', 0): |x'|\leq 1\}\subset\Gamma\subset\p\Omega.\] 
Let 
$$w=\Big [x_n -x_n^{\frac{2}{n+ p}} (1-|x'|^2)^{\frac{n}{n+ p}}\Big]/2$$ 
 be as in Lemma \ref{suplem}. 
 
We show that $w\geq u$ in $\Omega_1$. Note that $w=0\geq u$ on $\p\Omega_1$. 
If $w-u$ attains its minimum value on $\overline{\Omega_1}$ at $x_0\in\Omega_1$ with $w(x_0)<u(x_0)<0$, then $D^2 w(x_0)\geq D^2 u(x_0)$. It follows that
$$|w(x_0)|^{-p}\geq \det D^2 w(x_0)\geq \det D^2 u(x_0) \geq |u(x_0)|^{-p}.$$
Therefore, $|w(x_0)|^{-p}\geq |u(x_0)|^{-p}$ which contradicts $|w(x_0)|>|u(x_0)|$ and $p>0$.

Now, from $w\geq u$ in $\Omega_1$, we find that for $x=(0, x_n)\in\Omega_1$,
\[|u(x)|\geq |w(x)| =(x_n^{\frac{2}{n+ p}}-x_n)/2\geq x_n^{\frac{2}{n+ p}}/2= \dist^{\frac{2}{n-q}} (x,\p\Omega)/2, \]
if $x_n>0$ is sufficiently small.
This proves (\ref{up_est}).
\vglue 0.2cm
{\bf Case 2:} $q=0$ and $n\geq 3$. Applying Lemma \ref{supdist} to $\gamma=0$ and $m=1$, we find that
 \[|u(x)| \geq c(n,\Omega,\Gamma) \dist^{\frac{2}{n}}(x, \p\Omega) \]
for $x\in\Omega$ sufficiently close to $\Gamma$.
\vglue 0.2cm 

{\bf Case 3:} $0<q<n-2$ and $n\geq 3$. Fix $\beta\in (\frac{2}{n-q}, 1)$. 
 We show that for $x\in\Omega$ sufficiently close to $\Gamma$, 
 \begin{equation}
 \label{uqbeta}
 |u(x)|\geq c(n,q, \beta, \Omega,\Gamma)\dist^{\beta} (x,\p\Omega).\end{equation}
 By the convexity of $u$, we have (see \cite[p. 1531]{LSNS})
\[|u(x)|\geq \frac{\dist(x,\p\Omega)}{\text{diam}(\Omega)}\|u\|_{L^{\infty}(\Omega)}\quad\text{for } x\in\Omega.\]
 On the other hand, by \cite[Lemma 3.1(iii)]{LSNS}, we have 
$\|u\|_{L^{\infty}(\Omega)} \geq c(n, q,\Omega)>0$.
Thus, there is a constant $c_0(n, q,\Omega)>0$ such that
\begin{equation}
\label{case0}
|u(x)|\geq c_0 \dist(x,\p\Omega) \quad\text{for } x\in\Omega.\end{equation}
Define the sequence $\{\alpha_k\}_{k=0}^{\infty}$ by
 \[\alpha_0=1, \quad\text{and }\alpha_{k+1}= \frac{2+q\alpha_k}{n}\quad\text{for } k\geq 0.\] 
 Note that
 \[\alpha_{k+1} - \frac{2}{n-q} = \frac{q}{n}(\alpha_{k} - \frac{2}{n-q}).\]
 Thus, the sequence $\{\alpha_k\}$ is strictly decreasing from $1$ to $\frac{2}{n-q}$. 
 
 To prove (\ref{uqbeta}), it suffices to prove by induction that there is a constant $c_k(n, q, k, \Omega,\Gamma)>0$ such that for $x\in\Omega$ sufficiently close to $\Gamma$, we have 
 \begin{equation}
 \label{casek}
 |u(x)|\geq c_k\dist^{\alpha_k} (x,\p\Omega).\end{equation}
 The base case $k=0$ follows from (\ref{case0}). Suppose we have proved (\ref{casek}) for $k$. We need to prove it for $(k+1)$. Indeed, for $x\in\Omega$ sufficiently close to $\Gamma$, we have from the induction hypothesis that
 \begin{equation}
 \label{uckq}
 \det D^2 u\geq c_k^q \dist^{q\alpha_k} (x,\p\Omega).
 \end{equation}
 Fix $z\in \Omega$ being close to $\Gamma$. 
Then, by Lemma \ref{supdist} applied to the case $\gamma= q\alpha_k$ and $m= c_k^q$, there is a constant $c_{k+1}= c_{k+1}(n, k, q,\Omega,\Gamma)>0$ such that
\[|u(z)|\geq c_{k+1}\dist^{\alpha_{k+1}} (z,\p\Omega).  \]
Therefore,  (\ref{casek}) is proved for $(k+1)$.
\end{proof}

\section{Global log-Lipschitz estimates and application}
\label{subsup}

In this section, we prove Proposition \ref{Alek2d} and Theorem \ref{wphi}.
\begin{proof} [Proof of Proposition \ref{Alek2d}]
The proof is is based on a construction of a log-Lipschitz convex subsolution. Denote a point $x\in \R^n$ by $x=(x', x_n)$. 
Let $z=(z', z_n)$ be an arbitrary point in $\Omega$. By translation and rotation of coordinates, we can assume that:  $0\in \p\Omega$,  $\Omega\subset \R^n_+=\{x=(x', x_n)\in\R^n: x_n>0\}$, 
 the $x_n$-axis points inward $\Omega$, $z$ lies on the $x_n$-axis, and $z_n=\dist (z,\p\Omega)$. Let $D=\diam (\Omega)$, and
 $$v(x)=(M+2D^2) x_n \log (x_n/D) + x_n (|x'|^2-D^2).$$
Then, $v\leq 0$ on $\p\Omega$ and
\begin{equation*}
 D^2 v =
  \begin{pmatrix}
  2x_n & 0 & \cdots & 0 & 2x_1 \\
  0 & 2x_n & \cdots & 0 & 2x_2 \\
  \vdots  & \vdots  & \ddots & \vdots & \vdots  \\
  0 & 0 & \cdots & 2x_n &  2x_{n-1}\\
  2x_1 & 2x_2 & \cdots & 2x_{n-1} & \frac{M+ 2D^2}{x_n} 
 \end{pmatrix}.
\end{equation*}
By induction, we find that 
\[\det D^2 v(x)= 2^{n-1} x_n^{n-2} (M+ 2D^2-2|x'|^2)\geq 2M x_n^{n-2}\geq 2M \dist^{n-2}(x,\p\Omega).\]
Moreover, $v$ is convex in $\Omega$. From the assumption on $u$, we have $u\geq v$ on $\p\Omega$, and
\[\det D^2 v\geq \det D^2 u\quad\text{in }\Omega.\]
By the comparison principle in Lemma \ref{comp_prin}, we have $u\geq v$ in $\Omega$. Thus
 \begin{eqnarray*}|u(z)|\leq |v(z)|&=& (M+ 2D^2) z_n\log (D/z_n) + D^2 z_n\\ &\leq& C(M, D) z_n (1+ |\log z_n|).
 \end{eqnarray*}
 Since $z_n=\dist (z,\p\Omega)$, the proposition is proved.
\end{proof}
\begin{rem}\leavevmode
\begin{enumerate}
\item[(i)] The global Log-Lipschitz estimate in Proposition \ref{Alek2d} is sharp in two dimensions. This can be seen from formula (\ref{sigL}) for 
the solution $\sigma_T$ to (\ref{STET}) in a right triangle in the plane. 
\item[(ii)]
The global  log-Lipschitz estimate in  Proposition \ref{Alek2d} strengthens an almost Lipschitz estimate in the main theorem of \cite{LL}. See Theorem 1.1 there for the case $a=\infty$, $\alpha=0$ and $\beta = 2n-1$. 
\end{enumerate}
\end{rem}

\begin{proof}[Proof of Theorem \ref{wphi}] We divide the proof into several steps. 
Note that for $\gamma>0$, $\tilde u = \gamma u$ satisfies
\begin{equation}
\label{u_res}
\tilde U^{ij} D_{ij} \tilde w = -\gamma^{-1} f,\end{equation}
where $\tilde U=(\tilde U^{ij})$ is the cofactor matrix of $D^2 \tilde u$, and $\tilde w = (\det D^2 \tilde u)^{-1}$. Thus, 
by considering $\|f^{+}\|_{L^{\infty}(\Omega)} u$ instead of $u$, it suffices to prove the theorem under the assumption that \[\|f^{+}\|_{L^{\infty}(\Omega)}=1.\]

{\bf Step 1.} We first establish the lower bound on the Hessian determinant $\det D^2 u$ 
\begin{equation}
\label{detc1}
\det D^2 u \geq c_1(n,  \diam(\Omega))
\end{equation}
via the following upper bound for $w$: 
\begin{equation}
\label{wbound}
 w\leq C_1(n,  \diam(\Omega)).
\end{equation}
Indeed, we will use the Aleksandrov-Bakelman-Pucci estimate (see \cite[Theorem 9.1]{GT}) for (\ref{AMCE}). By this estimate, and $\det (U^{ij})= (\det D^2u)^{n-1}= w^{-(n-1)}$, we have
\begin{eqnarray*}
\norm{w}_{L^{\infty}(\Omega)}& \leq&  C(n) \diam (\Omega)\norm{\frac{f^{+}}{(\det U)^{1/n}}}_{L^{n}(\Omega)}\\
&=&  C(n) \diam (\Omega)\norm{f^{+} w^{\frac{n-1}{n}}}_{L^{n}(\Omega)}\leq  C(n)\diam (\Omega)\norm{w}^{\frac{n-1}{n}}_{L^{\infty}(\Omega)}.
\end{eqnarray*}
Thus, we can easily obtain (\ref{wbound}).

{\bf Step 2.} We show that
\begin{equation}\det D^2 u(x)\geq 
\begin{cases}
c(n, \diam (\Omega))  \dist^{-\frac{2}{n}}(x,\p\Omega)   & \quad\text{if }n>2,\\
c( \diam (\Omega)) \dist^{-1} (x,\p\Omega)  (1+ |\log \dist(x,\p\Omega)|)^{-1} & \quad\text{if } n=2.
\end{cases}
\label{wstep2n}
\end{equation}
Indeed, let $z=(z', z_n)$ be an arbitrary point in $\Omega$. By translation and rotation of coordinates, we can assume that:  $0\in \p\Omega$,  $\Omega\subset \R^n_+=\{x=(x', x_n)\in\R^n: x_n>0\}$, 
 the $x_n$-axis points inward $\Omega$, $z$ lies on the $x_n$-axis, and $z_n=\dist (z,\p\Omega)$. We will compare $w$ with  the functions $v_{\alpha}$ defined as follows. 

Let $D:=\diam(\Omega)$.  Consider for $\alpha\in [\frac{2}{n}, 1]$
\begin{equation}v_{\alpha}(x)= 
\begin{cases}
x_n^{\alpha} (|x'|^2 -C_\alpha)\quad\text{where } C_{\alpha}= \frac{1+ 2D^2 }{\alpha(1-\alpha)} & \quad \text{if } \alpha \in [\frac{2}{n}, 1), \\
(1+2D^2) x_n \log (x_n/D) + x_n (|x'|^2-D^2)& \quad \text{if } \alpha =1.
\end{cases}
\label{alphaeq}
\end{equation}
Then, for $\alpha\in [\frac{2}{n}, 1)$, we find from \cite[Lemma 2.2]{LDCDS} that $v_{\alpha}$ is convex in $\Omega$, $v_{\alpha}\leq 0$ on $\p\Omega$, and   
\begin{equation*}\det D^2v_{\alpha}= 2^{n-1}x_n^{n\alpha-2} [\alpha(1-\alpha) C_\alpha- (\alpha^2+ \alpha) |x'|^2] \geq x_n^{n\alpha-2}\geq \dist^{n\alpha-2}(x,\p\Omega)~\text{in}~\Omega.
\end{equation*}
For $\alpha=1$,  as in the proof of Proposition \ref{Alek2d}, we find that $v_1$ is convex in $\Omega$,
and \begin{equation*}\det D^2v_1\geq x_n^{n-2}\geq \dist^{n-2}(x,\p\Omega)~\text{in}~\Omega,~ \text{and} ~v_1\leq 0~\text{on}~\p\Omega.
\end{equation*}
Thus, for all $\alpha\in [\frac{2}{n}, 1]$, we have
$v_\alpha$ is convex in $\Omega$,
and \begin{equation}
\label{vadet}
\det D^2v_\alpha\geq \dist^{n\alpha-2}(x,\p\Omega)~\text{in}~\Omega,~ \text{and} ~v_\alpha\leq 0~\text{on}~\p\Omega.
\end{equation}

To prove (\ref{wstep2n}), it suffices to show that 
 \begin{equation}
 \label{beta1n2}
w\leq -C(n, \diam (\Omega)) v_{2/n}\quad\text{in }\Omega.
 \end{equation}
 Indeed, suppose (\ref{beta1n2}) holds. If $n>2$, then from (\ref{beta1n2}) and (\ref{alphaeq}), we have at $z=(0, z_n)$
 \begin{eqnarray*}w(z) \leq -C(n,\diam (\Omega)) v_{ 2/n}(z)&=& C(n, \diam (\Omega)) C_{2/n} z_n^{2/n}\\&=& C(n,\diam (\Omega)) \dist^{2/n} (z,\p\Omega).\end{eqnarray*}
 Therefore, (\ref{wstep2b}) holds in this case because $\det D^2 u(z)=[w(z)]^{-1}$. The case $n=2$ is similar.
 
 To prove (\ref{beta1n2}), 
 we will use
the matrix inequality: for positive $n\times n$ matrices $A$ and $B$, \begin{equation}
\label{Mineq}
\text{trace}(AB)\geq n (\det A)^{1/n} (\det B)^{1/n}.\end{equation}
Then, using (\ref{detc1}) and (\ref{vadet}) together with $\det (U^{ij})= (\det D^2 u)^{n-1}$, we find that
\begin{eqnarray*}U^{ij} D_{ij}(-\tilde Cv_{2/n})& \leq& -\tilde Cn  (\det D^2 u)^{\frac{n-1}{n}} (\det D^2 v_{2/n})^{\frac{1}{n}}\\ &\leq& -\tilde C n c_1^{\frac{n-1}{n}} <-1\leq  -f = U^{ij} D_{ij} w,
\end{eqnarray*}
if $\tilde C=\tilde C (n,\diam (\Omega)) $ is large. 

Since \[-U^{ij} D_{ij}(w+ \tilde C v_{2/n}) <0\quad \text{in } \Omega,\quad \text{and } w+ \tilde C v_{2/n} \leq 0 \quad \text{on } \p\Omega,\]
the classical comparison principle implies
$w +\tilde C v_{2/n}\leq 0$ in $\Omega$.

By Step 2, the theorem is proved for $n=2$. Thus, in the rest of the proof, we only consider $n>2$.

{\bf Step 3.} The conclusion of the theorem follows from Step 2 and  the following improvement on the blow up rate of the Hessian determinant: If for some $\beta\in [\frac{1}{n}, \frac{n-1}{n}]$, we have
\begin{equation}\det D^2 u(x)\geq C(n,\beta, \diam (\Omega))  \dist^{-\beta}(x,\p\Omega)  \quad\text{for all } x\in\Omega,
\label{wstep2}
\end{equation}
then 
\begin{equation}\det D^2 u(x)\geq 
\begin{cases}
c(n,\beta, \diam (\Omega))  \dist^{-\beta-\frac{1}{n}}(x,\p\Omega)   & \quad\text{if }\frac{1}{n}\leq \beta<\frac{n-1}{n},\\
c(n, \diam (\Omega)) \dist^{-1} (x,\p\Omega)  (1+ |\log \dist(x,\p\Omega)|)^{-1} & \quad\text{if } \beta=\frac{n-1}{n}.
\end{cases}
\label{wstep2b}
\end{equation}
Assume now (\ref{wstep2}) holds. We prove (\ref{wstep2b}).
Indeed, let $z=(z', z_n)$ be an arbitrary point in $\Omega$. By translation and rotation of coordinates, we can assume that:  $0\in \p\Omega$,  $\Omega\subset \R^n_+=\{x=(x', x_n)\in\R^n: x_n>0\}$, 
 the $x_n$-axis points inward $\Omega$, $z$ lies on the $x_n$-axis, and $z_n=\dist (z,\p\Omega)$.  Let $v_\alpha$ be as in (\ref{alphaeq}). To prove (\ref{wstep2b}), as in Step 2, it suffices to show that 
 \begin{equation}
 \label{beta1n}
w\leq -C(n,\beta, \diam (\Omega)) v_{\beta + 1/n}\quad\text{in }\Omega.
 \end{equation}
 Let us prove this inequality. Note that $\beta+1/n\geq 2/n$. Using (\ref{Mineq}), (\ref{wstep2}) and (\ref{vadet}) together with $\det (U^{ij})= (\det D^2 u)^{n-1}$, we find that
\begin{eqnarray*}U^{ij} D_{ij}(-\tilde Cv_{\beta+1/n})& \leq& -\tilde Cn  (\det D^2 u)^{\frac{n-1}{n}} (\det D^2 v_{\beta+1/n})^{\frac{1}{n}}\\ &\leq& -\tilde C C(n,\beta, \diam (\Omega))  \dist^{-\beta\frac{n-1}{n}}(x,\p\Omega) \dist^{\frac{n\beta-1}{n}}(x,\p\Omega)\\
&=&- \tilde C C (n,\beta, \diam (\Omega))  \dist^{\frac{\beta-1}{n}}(x,\p\Omega) <-1\leq  -f = U^{ij} D_{ij} w,
\end{eqnarray*}
if $\tilde C=\tilde C (n,\beta, \diam (\Omega)) $ is large. 

Applying
the classical comparison principle as in Step 2, we obtain
$w +\tilde C v_{\beta+1/n}\leq 0$ in $\Omega$.
This gives (\ref{beta1n}) and the theorem is proved.
\end{proof}
\section{Further remarks on surface tensions}
\label{sfrem_sect}
In general, the surface tension is a convex function 
on a convex polygon in the plane; it solves a Monge-Amp\`ere equation with constant right-hand side, except possibly a finite number of singularities, and has 
piecewise affine boundary value. Understanding surface tensions and variational problems associated with them lead to deeper understanding of limit shapes of random surfaces; see, for example \cite{ADPZ, DS, KO2}.  
In this section, we show that the gradient of general surface tensions behaves similarly to that of the lozenge tiling model described in Section \ref{lozenge_sect}. Below, we describe below more precisely the context.
\vglue 0.2cm
Let $N\subset\R^2$ be a compact convex polygon with interior $N^o$. Let $L:N\rightarrow\R$ be a convex function which is piecewise affine on the boundary $\p N$. 
Let $\mathscr{P}$ be the set of corners of $N$, and let $\mathscr{L}\subset \p N\setminus \mathscr{P}$ be the finite set of ``quasi frozen" points where $L$ is not differentiable. Let $\mathscr{G}=\{q_1,\cdots, q_l\}\subset N^0$  be a finite set of points; when $l=0$, we make the convention that $\mathscr{G}=\emptyset$. The set $\mathscr{G}$ is called the set of ``gas" points. Let $c_1,\cdots, c_l$ be positive numbers.

We consider the Monge-Amp\`ere equation
\begin{equation}
 \label{STE}
 \left\{
 \begin{alignedat}{2}
   \det D^{2} \sigma~&= 1+\sum_{k=1}^l c_k \delta_{q_k} \h~&&\text{in} ~N^o, \\\
\sigma &=L\h~&&\text{on}~\p N.
 \end{alignedat}
 \right.
\end{equation}
The function $\sigma$, which solves (\ref{STE}) in the sense of Alkesandrov, is called the surface tension. Due to the appearance of the Dirac masses at $\{q_k\}_{k=1}^l$, (\ref{STE}) is in fact a singular Monge-Amp\`ere equation when $l\geq 1$.
\vglue 0.2cm
By \cite[Theorem 1.1]{H}, there is a unique convex solution $\sigma\in C(N)$ to (\ref{STE}).
It was proved in a recent paper of Astala, Duse, Prause, and Zhong \cite{ADPZ} that for any closed interval $J\subset \p N$ not containing any points of $\mathscr{P}\cup\mathscr{L}$, we have
\[\lim_{z\rightarrow J, z\in N^o} |D\sigma(z)|\rightarrow \infty.\]
Here, we give the precise blow up rate of $|D_\nu\sigma|$ near $J$ where $\nu$ is the unit vector perpendicular to $J$ and $D_\nu \sigma= D\sigma \cdot \nu$. Our result states as follows.
\begin{prop} 
\label{Dsiglog}
Let $\sigma\in C(N)$ be the convex solution to (\ref{STE}).  Consider a closed interval $J\subset \p N$ not containing any points of $\mathscr{P}\cup\mathscr{L}$. Let $\nu$ be the unit vector perpendicular to $J$. Then, 
there exists positive constants $c$ and $C$, depending only on $N, J, L, \{c_{k}\}_{k=1}^l$, and $\mathscr{G}$ such that
\[c |\log \dist(z, \p N|\leq |D_\nu\sigma (z)| \leq C |\log \dist(z, \p N|\]
when $z\in N^o$ with $\dist (z, J)$ being sufficiently small.
\end{prop}

The proof of Proposition  \ref{Dsiglog} is based on an explicit formula (\ref{sigL}) for the surface tension $\sigma_T$ in the lozenge tiling model where $N$ is a triangle $T$ and there are no gas points, and Proposition \ref{Alek2d}. To handle the gas points $\{q_k\}$, we compare $\sigma$ with linear combinations of $\sigma_T$ and  conical convex functions $\hat C_{q_k, N}$ taking value $0$ on $\p N$ and $-1$ at $q_k$. 
 These conical convex functions are globally Lipschitz
 so the boundary gradient of $\sigma$ has the magnitude comparable to that of $\sigma_T$. For this, we use the following property of Monge-Amp\`ere measures; see \cite[Lemma 2.9]{F}.
\begin{lem}
\label{detuv}
 Let $u$ and $v$ be convex functions on a convex domain $\Omega\subset\R^n$. Then, in the sense of Aleksandrov, we have
 $$\det D^2(u+ v)\geq \det D^2 u + \det D^2 v.$$
 \end{lem}

Now, we proceed to prove Proposition \ref{Dsiglog}. 

We denote the points in $\mathscr{P}\cup \mathscr{L}$ as $\{p_1, \cdots, p_k\}$ in the counterclockwise direction. For convenience, denote $p_{k+1}= p_1$. Let $l_j$ denote the line going through $p_j$ and $p_{j+1}$. 
Let $L|_{[p_j, p_{j+1}]}$ be the restriction of $L$ to the line segment \[[p_j, p_{j+1}]:=\{tp_j + (1-t) p_{j+1}: 0\leq t\leq 1\}.\] 
Due to the convexity of the boundary data $L$ and its being affine on each segment $[p_j, p_{j+1}]$, we have the following lemma which says that $L$ is one-sided Lipschitz on the boundary.
\begin{lem}[One-sided Lipschitz property of boundary data]
\label{Llog}
Let $L_j:=L- L|_{[p_j, p_{j+1}]}$. Then, for all $x\in \p N$, we have
\begin{equation}\label{Lldist} L_j(x)\geq -C_0(L, N) \dist (x, l_j).
\end{equation}
\end{lem}
\begin{proof}
The line segment $[p_j, p_{j+1}]$ lies on some segment $[X, Y]\subset\p N$ where $X$ and $Y$ are vertices of $N$. By a coordinate change, we can assume that $N$ lies in the upper-half plane $\{x=(x_1, x_2)\in \R^2: x_2\geq 0\}$, and that $[X, Y]$ lies on the $x_1$-axis so $l_j$ is the $x_1$-axis. Note that $L_j=0$ on $[p_j, p_{j+1}]$. There are three cases.

{\bf Case 1:} $[p_j, p_{j+1}]$ lies in the interior of $[X, Y]$. By the convexity of $L_j$,  and $L_j=0$ on $[p_j, p_{j+1}]$, we have \[L_j>0 \quad\text{in }[X, Y]\setminus [p_j, p_{j+1}].\]
If $x\in \p N$ lies on a side $S$ with no common point with $XY$ then $x_2\geq c(N)>0$ so (\ref{Lldist}) is obvious. It remains to consider the case $S$ has a common point with $XY$, say $Y$. Since $L_j(Y)>0$, and $L_j$ is continuous on $S$, there is $c_j>0$ such that if $x\in S$ with $x_2\leq c_j$, then $L_j(x)\geq 0$. Then (\ref{Lldist}) also holds for $x\in \p N$ with $x_2>c_j$.

{\bf Case 2:} $[p_{j}, p_{j+1}]= [X, Y]$.   As in {\it Case 1}, it suffices to consider the case $x\in\p N$ where $S$ is a side of $N$ having a common point with $[X, Y]$. Without loss of generality, we can assume that common point is $Y$ and, by a translation of coordinates, $Y=0$. Since $L_j(0)=0$, we have $L_j(x) = a_1x_1 + a_2x_2$ on $S$. Since $S$ is not horizontal, it is a graph of the $x_2$ variable so there is $m\in\R$ such that $x_1=m x_2$ on $S$. It follows that
\[L_j(x)= (a_1 m + a_2) x_2 \geq -C x_2\]
which proves (\ref{Lldist}). 

{\bf Case 3:} One of $p_i, p_{j+1}$ is an endpoint of $[X, Y]$, say $p_{j+1}= Y$.  As in {\it Case 1}, it suffices to consider the case $x\in\p N$ where $S$ is a side of $N$ having a common point with $[X, Y]$. If $X\in S$, then we argue as in {\it Case 1} to obtain (\ref{Lldist}). If $Y\in S$, then we argue as in  {\it Case 2} to obtain (\ref{Lldist}).
\end{proof}

\begin{proof}[Proof of Proposition \ref{Dsiglog}] Assume that $J\subset (p_{j}, p_{j+1})$. 
By subtracting the affine function $L|_{[p_j, p_{j+1}]}$ from $\sigma$, we can assume that $\sigma=0$ on $J$.  This will not change our estimates for $|D_\nu \sigma|$ since this affine function has bounded slope which will be absorbed in the $\log$ terms. Let $T\subset\R^2$ be the triangle with vertices at $(0,0),  (1,0)$ and $(0, 1)$.

From the hypothesis on $J$, we know that $\det D^2\sigma =1$ and $\sigma$ is smooth in a neighborhood of $J$ in $N^o$. Recall that the Monge-Amp\`ere equation is invariant under the affine transformations of coordinates where the determinant of the transformations is $1$.

Thus, after a rotation and an affine transformation of coordinates, we can assume that $N$ is in the upper half space $\{x=(x_1, x_2)\in\R^2: x_2\geq 0\}$, and $J\subset\p N$ lies on the $x_1$ axis. From the assumption on $J$, we can find small constants $0<r<1$, $0<\delta<1/2$, and a positive integer $s$, all depending on $J, N, \mathscr{P}, \mathscr{L}, \mathscr{G}$ such that
\[J\subset \bigcup_{k=1}^s r\big([\delta, 1-\delta] \times \{0\}+ w^k\big),\quad \bigcup_{k=1}^s r\big([0, 1]  \times \{0\} + w^k\big) \subset [p_j, p_{j+1}],\]
for suitable $w^k=(w^k_1,0)\in\R^2$, where
\[ \bigcup_{k=1}^s r[T + w^k]\subset N, \quad \text{and}\quad \bigcup_{k=1}^s r[T + w^k]\quad\text{does not intersect }\mathscr{P}\cup \mathscr{L}\cup\mathscr{G}.\]
To prove the Proposition, it suffices to consider the case $s=1$, and $w^k=0$.
More precisely, we only need to consider
 the following scenario: 
\begin{center}
\it
There are $0<r<1$ and $0<\delta<1/2$ such that 
 $r T\subset N$ and $r T$ does not intersect $\mathscr{P}\cup \mathscr{L}\cup\mathscr{G}$, 
$J$ is contained in the interval $r[\delta, 1-\delta]\times \{0\}$ of the $x_1$-axis, and $r[0, 1]\times \{0\}\subset [p_j, p_{j+1}]$.
\end{center} 
We will establish two separate upper and lower bounds on $|D_\nu\sigma(r\gamma, z_2)|= |\sigma_{x_2}(r\gamma, z_2)|$
 for the points $(r\gamma, z_2)$ being close to $J$ where $\delta\leq \gamma\leq 1-\delta$ and $z_2>0$ is small. Note that, in this setting $ |\sigma_{x_2}(r\gamma, z_2)|=-\sigma_{x_2}(r\gamma, z_2)$.

{\bf Step 1: Upper bound on $ |\sigma_{x_2}(r\gamma, z)|$.} For each $k\in\{1, \cdots, l\}$, consider the conical convex function $\hat C_{q_k, N}$ which takes value $0$ on $\p N$ and $-1$ at $q_k$. Then $\hat C_{q_k, N}$ is globally Lispchitz with 
\begin{equation}
\label{CqkN}
|D\hat C_{q_k, N}|\leq C(q_k, N).\end{equation}
Then, there is a positive constant $a_k$ depending on $N, q_k$ and $c_k$ such that
\begin{equation}
\label{coniC}
\det (D^2 (a_k \hat C_{q_k, N})) = c_k \delta_{q_k}.
\end{equation}

 Let $C_0$ be as in Lemma \ref{Llog} and let $d=\diam (N)$. As in the proof of Proposition \ref{Alek2d}, we find that the function
\begin{equation}
\label{vNeq}
v(x)=(1+2d^2) x_2 \log (x_2/d) + x_2 ((x_1-r\gamma)^2-d^2-C_0)
\end{equation}
 is convex in $N$ with 
 \begin{equation}
 \label{vbdr1}
 v\leq -C_0 x_2 \quad\text{on }\p N, \quad\text{and } 
\det D^2 v\geq 2\quad\text{ in }N.
\end{equation}
Observe that we can enclose $N$ in a triangle of the form $sT + (a,0)$ where $s$ is large, and then use the formula (\ref{sigL}) for $\sigma_T$ to construct a convex function $v$ satisfying (\ref{vbdr1}). Our construction of $v$ in (\ref{vNeq}) is more direct. 

From $L|_{[p_j, p_{j+1}]}=0$, using Lemma \ref{Llog}  and (\ref{vbdr1}), we obtain on $\p N$ the estimate
\[\sigma = L= L_j \geq v  +  \sum_{k=1}^l  a_k \hat C_{q_k, N}. \]
Moreover, in view of Lemma \ref{detuv}, the second inequality in (\ref{vbdr1}) together with (\ref{coniC}) gives 
\[\det\Big[D^2 ( v +  \sum_{k=1}^l  a_k \hat C_{q_k, N})\Big] \geq 2  + \sum_{k=1}^l c_k \delta_{q_k}\geq \det D^2 \sigma. \]
It follows from the comparison principle in Lemma \ref{comp_prin} that 
\begin{equation}
\label{siglow1}
\sigma\geq v  + \sum_{k=1}^l  a_k \hat C_{q_k, N}.
\end{equation}
In particular, we have
\begin{equation}
\label{siglow2}
\sigma\geq -C(\mathscr{G}, L, N, \{c_k\}_{k=1}^l).
\end{equation}
By convexity, we have \[0=\sigma(r\gamma,0)\geq \sigma (r\gamma, z_2) + \sigma_{x_2} (r\gamma, z_2) (0- z_2).\] Hence, recalling (\ref{CqkN}) and (\ref{vNeq}), we have
\begin{eqnarray*}
-\sigma_{x_2} (r\gamma, z_2) &\leq& \frac{-\sigma (r\gamma, z_2)}{z_2}\\
&\leq& \frac{ - v(r\gamma, z_2)}{z_2} - \sum_{k=1}^l  a_k  \frac{ \hat C_{q_k, N}(r\gamma, z_2)}{z_2} 
\\&\leq& -(1+ 2d^2)\log (z_2/d) + d^2+ C_0+  \sum_{k=1}^l  a_k C(q_k, N)\\&\leq& C(N, J, L, \{c_{k}\}_{k=1}^l, \mathscr{G})  |\log z_2|
\end{eqnarray*}
 when $z_2>0$ is small. This gives the desired upper bound for $|\sigma_{x_2}(r\gamma, z_2)|$.

{\bf Step 2: Lower bound on $|\sigma_{x_2}(r\gamma, z_2)|$.}
This is an elaboration of the proof of Lemma 2.1 in \cite{ADPZ} in combination with the  global lower bound for $\sigma$ in (\ref{siglow1}). Let $h\in C(N)$ be the harmonic function in $N^0$ with boundary value $L$ on $\p N$. Then, since $\sigma$ is subharmonic, we have 
\begin{equation}
\label{sigh}
\sigma\leq h\quad\text{in } N.
\end{equation}
Combining (\ref{siglow2}) with (\ref{sigh}), we obtain
\begin{equation}
\label{sigbound}
|\sigma| \leq C(\mathscr{G}, L, N,  \{c_k\}_{k=1}^l).
\end{equation}

Consider the linear map \[\tilde L(x_1, x_2) = \sigma ((0, r)) x_2/r.\] Then $\tilde L=\sigma=0$ on $[0, r]\times \{0\}$ due to $\sigma=0$ on $[p_j, p_{j+1}]$, and
$\tilde L(0,r)=\sigma (0,r)$. Then, by convexity, we have $\sigma\leq \tilde L$ in $rT$. Let $\sigma_T$ be as in (\ref{sigL}). Then \[\det D^2 ( r^2\sigma_T(x/r))= 1\quad\text {in }rT.\]
Using the comparison principle in Lemma \ref{comp_prin}, we have
 \begin{equation}
 \label{sigup}
 \sigma \leq r^2\sigma_T(x/r) + \tilde L\quad\text{in } rT.
 \end{equation}
Let $m>1$ be a large constant to be determined.  Consider 
$z_2>0$ small so that $(r\gamma, (m+1)z_2)\in rT$. By convexity, we have
\begin{eqnarray*}
-\sigma_{x_2} (r\gamma, z_2) \geq \frac{\sigma(r\gamma, z_2)-\sigma(r\gamma, (m+1) z_2)}{m z_2}
\end{eqnarray*}
Therefore, recalling (\ref{siglow1}) together with (\ref{CqkN}), and using (\ref{sigup}) for $\sigma(r\gamma, (m+1) z_2)$, we can estimate
\begin{eqnarray}
\label{sigx21}
-\sigma_{x_2} (r\gamma, z_2) &\geq& \frac{(1+ 2d^2)\log (z_2/d) -d^2-C_0 - \sum_{k=1}^l a_k C(q_k, N)} {m} 
\nonumber\\&& -\frac{r^2 \sigma_T(\gamma, z_2 (m+1)/r)}{mz_2} - \frac{(m+1)\sigma(0, r)}{mr}
\end{eqnarray}
When $z_2>0$ is small, we have, in view of (\ref{Dsig})
\begin{eqnarray}
\label{sigx22}
-\frac{r^2 \sigma_T(\gamma, z_2 (m+1)/r)}{mz_2}&\geq& - \frac{1}{2}\frac{r (m+1)}{m} \sigma_{T, x_2} (\gamma,z_2 (m+1)/r)
\nonumber\\&\geq& c \frac{r (m+1)}{m}  |\log z_2|.
\end{eqnarray}
where $c>0$ is a positive constant.

Combining (\ref{sigx21}) with (\ref{sigx22}) and (\ref{sigbound}), we obtain 
the desired lower bound for  $|\sigma_{x_2}(r\gamma, z_2)|$ when $m$ is large and $z_2$ is small. 

The proof of the Proposition is complete.
\end{proof}


\begin{thebibliography}{9999999}
\bibitem[Ab]{Ab} Abreu, M. K\"ahler geometry of toric varieties and extremal metrics, {\it Inter. J. Math.} {\bf 9} (1998), 641-651.
\bibitem[A]{Ahlf} Ahlfors, L. V. {\em Complex analysis. An introduction to the theory of analytic functions of one complex variable.} Third edition. International Series in Pure and Applied Mathematics. McGraw-Hill Book Co., New York, 1978.
\bibitem[ADPZ]{ADPZ}Astala, K.; Duse, E.; Prause, I.; Zhong, X. Dimer Models and Conformal Structures,  arXiv:2004.02599v2.
\bibitem[C]{C} Caffarelli, L. A. A localization property of viscosity solutions to the Monge-Amp\`ere equation and their strict convexity. {\it Ann. of Math.} (2) {\bf 131} (1990), no. 1, 129--134. 

\bibitem[CLS]{CLS}Chen, B. ; Li, A-M. ; Sheng, L.  The Abreu equation with degenerated boundary conditions. {\it J. Differential Equations} {\bf 252} (2012), no. 10, 5235--5259.
\bibitem[CKP]{CKP}Cohn, H.; Kenyon, R.; Propp, J. A variational principle for domino tilings. {\it J. Amer. Math. Soc.} {\bf 14} (2001), no. 2, 297--346. 
\bibitem[DS]{DS} De Silva, D.; Savin, O. Minimizers of convex functionals arising in random surfaces. {\it Duke Math. J.} {\bf 151} (2010), no. 3, 487--532.
 \bibitem[D]{D} Du, S.Z. On Bernstein Problem of Affine Maximal Hypersurfaces, arXiv:2210.09127 [math.DG].
\bibitem[D1]{D1}Donaldson, S. K. Scalar curvature and stability of toric varieties.  {\it J. Differential Geom.}  {\bf 62}  (2002),  no. 2, 289--349.
\bibitem[D2]{D2} Donaldson, S. K. Interior estimates for solutions of Abreu's equation.  {\it Collect. Math.}  {\bf 56}  (2005),  no. 2, 103--142.
\bibitem[F]{F} Figalli, A. {\em The Monge-Amp\`ere equation and its applications.}  Zurich Lectures in Advanced Mathematics. European Mathematical Society (EMS), 2017.
\bibitem[GT]{GT} Gilbarg, D.; Trudinger, N. S. {\it Elliptic partial 
differential equations of second order}. Reprint of the 1998 edition. Classics in Mathematics. Springer-Verlag, Berlin, 2001.
\bibitem[Go]{Go}Gorin, V. {\em Lectures on random lozenge tilings.} Cambridge Studies in Advanced Mathematics, 193. Cambridge University Press, Cambridge, 2021.
 \bibitem[Gu]{G} Guti\'errez, C. E. {\em The Monge-Amp\`ere equation.} Second edition.  Progress in Nonlinear Differential Equations and their Applications, 89. 
  Birkha\"user, Boston, 2016.
\bibitem[H]{H}Hartenstine, D. The Dirichlet problem for the Monge-Amp\`ere equation in convex (but not strictly convex) domains. {\it Electron. J. Differential Equations} (2006), no. 138, 9 pp.
\bibitem[K]{K} Kenyon, R. Lectures on dimers. In: {\em Statistical mechanics}, 191--230, IAS/Park City Math. Ser., 16, Amer. Math. Soc., Providence, RI, 2009.
\bibitem[KO1]{KO} Kenyon, R.; Okounkov, A. Planar dimers and Harnack curves. {\it Duke Math. J.} {\bf 131} (2006), no. 3, 499--524.
\bibitem[KO2]{KO2} Kenyon, R.; Okounkov, A. Limit shapes and the complex Burgers equation. {\it Acta Math.} {\bf 199} (2007), no. 2, 263--302.
\bibitem[KOS]{KOS}Kenyon, R.; Okounkov, A.; Sheffield, S.
Dimers and amoebae. 
{\it Ann. of Math. (2)} {\bf 163} (2006), no. 3, 1019--1056.
 \bibitem[L1]{LSNS} Le, N. Q. The eigenvalue problem for the Monge-Amp\`ere operator on general bounded convex domains,  
 {\it Ann. Sc. Norm. Super. Pisa Cl. Sci.} (5) {\bf 18} (2018), no. 4, 1519-1559. 
\bibitem[L2]{LDCDS}Le, N. Q. Optimal boundary regularity for some singular Monge-Amp\`ere equations on bounded convex domains. {\it Discrete Contin. Dyn. Syst.}. {\bf 42} (2022), no. 5, 2199-2214.
\bibitem[LMT]{LMT} Le, N. Q.; Mitake, H.; Tran, H. V. {\em Dynamical and geometric aspects of Hamilton-Jacobi and linearized Monge-Amp\`ere equations--VIASM 2016}. Edited by Mitake and Tran. Lecture Notes in Mathematics, 2183. Springer, Cham, 2017.
\bibitem[LL]{LL} Li, M.; Li, Y.; Global regularity for a class of Monge-Amp\`ere type equations. {\it Sci. China Math.} {\bf 65} (2022), no. 3, 501--516.
\bibitem[Tso]{Tso}Tso, K. On a real Monge-Amp\`ere functional. {\it Invent. Math.} {\bf 101} (1990), no. 2, 425--448.



\end{thebibliography}
\end{document}